\documentclass[aip,jcp,imp,amsmath,amssymb,reprint]{revtex4-1}
\usepackage{amsmath}
\usepackage{graphicx}
\usepackage{dcolumn}
\usepackage{amssymb}
\usepackage{epsfig}
\usepackage{wrapfig}
\newtheorem{theorem}{Theorem}[section]

\newenvironment{proof}[1][Proof]{\begin{trivlist}
\item[\hskip \labelsep {\bfseries #1}]}{\end{trivlist}}

\newcommand{\qed}{\nobreak \ifvmode \relax \else
      \ifdim\lastskip<1.5em \hskip-\lastskip
      \hskip1.5em plus0em minus0.5em \fi \nobreak
      \vrule height0.75em width0.5em depth0.25em\fi}
\begin{document}

\newcommand{\mps}{|\Psi\rangle}
\newcommand{\ha}{\frac{1}{2}}

\title{Dimensional Reductions for the Computation of Time--Dependent Quantum Expectations} 
\author{Giacomo Mazzi}
\email[Electronic Mail: ]{G.Mazzi@sms.ed.ac.uk}
%\homepage[]{Your web page}
%\thanks{}
%\altaffiliation{}
\author{Benedict J. Leimkuhler}
\affiliation{School of Mathematics, University of Edinburgh, Edinburgh EH9 3JZ, UK}
\date{\today}
\begin{abstract}
We consider dimensional reduction techniques for the Liouville-von Neumann equation for the evaluation of the expectation values in a mixed quantum system. In applications such as nuclear spin dynamics the main goal for simulations is being able to simulate a system with as many spins as possible, for this reason it is very important to have an efficient method that scales well with respect to particle numbers.  We describe several existing methods that have appeared in the literature, pointing out  their limitations particularly in the setting of large systems.
We introduce a method for direct computation of expectations via Chebyshev polynomials (DEC) based on evaluation of a trace formula combined with expansion in modified Chebyshev polynomials. This reduction is highly efficient and does not destroy any information.  We demonstrate the practical application of the scheme for a nuclear spin system and compare with several alternatives, focusing on the performance of the various methods with increasing system dimension. Our method may be applied to autonomous quantum problems where the desired outcome of quantum simulation, rather than being a full description of the system dynamics,  is only the expectation value of some given observable.

\end{abstract}

%\begin{keywords} 
% Chebyshev Expansion, Krylov Subspace, Density Matrix, Quantum Mechanics,  Nuclear Magnetic Resonance Simulation, Nuclear Spin Dynamics.
%\end{keywords}
\maketitle
\section{Introduction}

The motion of a quantum system of particles is described by the Schr\"odinger equation :
\begin{equation}\label{schro0}
 \frac{\partial \mps}{\partial t}=-i H \mps, 
\end{equation}
where $H$ is the Hamiltonian and $\mps$ the wave function of the system, and we have chosen physical quantities such that $\hbar=1$.
When $H$ is time independent it is possible to have an exact representation of the solution:
\begin{equation}\label{schro}
|\Psi(t)\rangle=e^{-iH t}|\Psi_0\rangle.
\end{equation}
From a numerical point of view finding a solution of (\ref{schro}) is a challenging task, due to the cost of computing the operator exponential $e^{-i H t}$.
Many numerical methods have been proposed over the years to solve (\ref{schro}), from polynomial expansion of the exponential \cite{tal,vijay,spinev}, to projection on Krylov subspaces \cite{moro,park}. When the wave function $\mps$ depends on the particle positions $q$ and momenta $p$ it is very common to adopt a splitting method for the Hamiltonian, $H=T+V$ being  $V$ diagonal and $T$ the kinetic term which is diagonal in Fourier space. With this approach the main issue is the evaluation of a fast Fourier transform (FFT) to switch between the $x$-grid and the $p$-grid; the cost of this transform is of order $O(N\log N)$, \cite{hardin}. 
However when the wave function depends on other variables, like the spin of the particles, it is not straightforward to apply a splitting method.
For a general survey of these approaches see Refs.~\onlinecite{leforestier,lubichbok}.

It is important to remark that all of these methods suffer when the dimension of the space  needed to represent the wave function $\Psi$  accurately becomes large.
In many practical cases, as for example in spin dynamics, the complexity of the computational problem grows exponentially with the number of atoms. For pure spin dynamics $\Psi$ may be represented by a vector in a space of dimension $(2I+1)^n$, where $n$ is the number of particles with spin$I$ present.

This paper is organized as follows:  in \S2 we introduce the Liouville--von Neumann equation that describes the dynamics of the density matrix, we also clarify the issues that complicate its numerical solution. In \S3 we discuss two of the main techniques that have been applied in quantum simulation for the evaluation of the exponential that appears in the solution of the Schrodinger equation, specifically Krylov expansion via Lanczos--Arnoldi iterations and the expansion in Chebyshev polynomials.   We also present a more recent method based on a different evaluation of the Krylov subspace expansion called Zero Track Elimination (ZTE) \cite{kuprov}. This method  has been developed for simulations of nuclear spin dynamics within the Nuclear Magnetic Resonance (NMR) community. We give what we believe to be the first mathematical discussion of the method, exploring its limitations and clarifying situations in which it works.

In \S4 we introduce a new method we have developed: the Direct Expectation values via Chebyshev (DEC). This is the main result of this paper.   In this section we also discuss the technical details of the implementation of the DEC method.

In \S5 we provide a performance comparison of DEC with three separate alternatives: the classic Krylov expansion based on Lanczos algorithm, with the new ZTE method and with the Chebyshev expansion. The sample system considered is a pairwise interacting nuclear spin system; even if it is a simple system it contains the key features of the systems of interest to the NMR community.

\section{The Liouville von--Neumann Equation}
When the sample of interest is composed of a large number $N$ of identical systems each described by a wave function $|\Psi^j \rangle$, $j=1,\ldots,N$, we may introduce a statistical operator called density operator $\varrho$ defined by:
\begin{equation}
\varrho=\sum_{j=1}^N |\Psi^j\rangle\langle\Psi^j|,
\end{equation}
which evolves according to the Liouville-von Neumann equation:
\begin{equation}\label{liovvn}
 \frac{\partial \varrho}{\partial t}=-i[H,\varrho].
\end{equation}
If $\varrho$ is expanded using a (finite) approximate basis set $\{|\varphi_1\rangle\ldots |\varphi_n\rangle\}$,  (\ref{liovvn}) may be viewed as an ordinary differential equation in a matrix argument.
In the time independent case the solution for (\ref{liovvn}) may be written:
\begin{equation}\label{sol}
 \varrho(t)=e^{-iH t}\varrho_0 e^{iH t}.
\end{equation}

It is possible and sometimes preferable to rewrite (\ref{sol}) by use of the 
Liouvillian $L= \rm{Id} \otimes H-H \otimes \rm{Id}$, where $\rm{Id}$ is the identity matrix,  allowing us to recast $\varrho$ as a vector:
\begin{equation}\label{liov}
\varrho(t)=e^{-iL t}\varrho_0.
\end{equation}
The main issue is then to evaluate the exponential of the matrix $L$.
For a large system and using a typical basis set, $H$ (and consequently $L$) will be a structured sparse matrix.

While the evolution of the system is described by the density matrix, the outputs we are  typically interested in obtaining from quantum simulations are
the expectations of observables, these being the only quantities we can compare with the experiments.
In the density matrix formalism the expectation value of an observable $Q$, associated with an operator $\hat{Q}$ is written as:
\begin{equation}\label{obs}
 \langle \hat{Q}(t) \rangle={\rm Trace} \{\varrho(t) \hat{Q}\}.
\end{equation}
Whereas in general quantum simulations the equation of motion is solved for a quantity, $\varrho$, which has dimension $n \times n$, 
the types of outputs we are actually interested in are typically one dimensional objects (\ref{obs}).
In this paper we exploit this fact and design an algorithm that computes (almost) directly the evolution of the expectation value (\ref{obs}),
instead of the evolution of the density matrix (\ref{sol}).
This approach does not lose any information of the original system (the only errors arise due to truncation), but at the same time the method
provides a powerful computational tool,  with potential dramatic reduction in the computational cost, especially when dealing with large matrices.
The main idea of this approach is to exploit features of a Chebyshev expansion for the matrix exponential in (\ref{liov}). 

Several methods to evaluate (\ref{liov}) are discussed in a recent monograph \cite{lubichbok}; however our  proposed {\em Direct evaluation of the Expectation values via Chebyshev polynomials} (DEC) method  is different because it does not solve the evolution for the density matrix. Instead it exploits the trace evaluation in (\ref{obs}) and with the evaluation of just one Chebyshev expansion it allows the solution of (\ref{obs}) at any time. 
DEC can be extremely powerful when we are only interested in the expectation values; if instead it is necessary to evaluate the evolution of the density matrix itself (\ref{liov}), then a traditional approach might be the best choice.

\section{Existing Methods}

\subsection{Krylov expansion via Lanczos--Arnoldi}
In the past decades many methods have been developed for numerical evaluation of the matrix exponential, \cite{nineteen}.
However for large sparse matrices, the methods usually applied are those based on expansion in Krylov subspace \cite{saad1,lubich1,schulze,bai}.
The main idea is to project (\ref{liov}) onto the subspace:
\begin{equation}\label{kry_sub} 
K_m(L,\varrho)={\rm span} \{\varrho_0,L\varrho_0,L^2\varrho_0,\ldots,L^m\varrho_0\}.
\end{equation}
To get a suitable basis set of (\ref{kry_sub}) we may use the Lanczos algorithm \cite{lancz}, as  $L$ is Hermitian.
The Lanczos method is an iterative method, a very desirable feature in the context of large sparse matrices.
For specific details see Ref.\cite{golubok,cullum}.

An approximation for $\varrho(t)$ is:
\begin{equation}\label{kryrho}
\varrho(t)=e^{-i L t}\varrho_0 \approx \|\varrho_0\| V_m e^{-i T_m t} e_1, 
\end{equation}
where $T_m$ and $V_m$ come from the Lanczos algorithm. The Lanczos algorithm provides an orthonormal basis set $V_m$ for the Krylov subspace $K_m(L,\varrho_0)$ via a three-term recursion \cite{lancz, cullum}:
\begin{equation}\label{lancz}
 \beta_{j+1}q_{j+1}=Lq_j-\alpha_jq_j-\beta_j q_{j-1} \qquad q_1=\varrho_0.
\end{equation}
$T_m$ is a tridiagonal matrix of size $m \times m$, and $e_1$ is the first vector of the canonical basis of size $n$.
This technique is very powerful for short time simulations, because with few iterations $m$ it is possible to have remarkably good approximations, but for longer times larger Krylov subspaces would be needed to stay close to the real solution. On the other hand if we do not consider enough terms in the Lanczos algorithm for longer times, (\ref{kryrho}) is no longer a reliable approximation. 

It is possible to set a stopping criterion for the Lanczos iterations \cite{lubich1}; for a given $t$ we can find $m$ such that:
\begin{equation}\label{stopcrit}
t [T_m]_{m+1,m} \|e^{-i tT_m}\|_{m,1}\leq \varepsilon.
\end{equation}

In our experiments the best way to implement a Krylov expansion was to evaluate each step:
 \begin{equation}
  \varrho_{n+1}=e^{-iL dt}\varrho_n\approx \|\varrho_n \| V_m^n e^{-i T_m^n t} e_1.
 \end{equation}
In this way with less than $10$ iterations of Lanczos per step it is possible to have a fast and accurate benchmark. The obvious drawback is that no information passes from step $n$ to step $n+1$. However in our numerical tests the use of a longer timestep that would allow a common Krylov subset $K_m(L,\varrho)$ for more than one step $\varrho_n$ was not preferable as more iterations of (\ref{lancz}) were then required.   As the equation (\ref{stopcrit}) involves the evaluation of the exponential of a tridiagonal matrix, when $m$ becomes large, this operation may become a serious bottleneck for the whole simulation. 

\subsection{The Chebyshev expansion}\label{cheby}
Another method that has been widely applied for solving (\ref{schro0}) is the expansion of the exponential in Chebyshev polynomials \cite{tal,vijay}.

The preliminary step of this method is to rescale the matrix within the interval $[-1,1]$, as outside this interval the Chebyshev polynomials grow rapidly, and the expansion becomes unstable; to do that we need to evaluate the two extremes of the spectrum of $L$.   
In order to obtain extreme values we propose, as already mentioned in the literature \cite{saad_lanc}, to perform a few steps of Lanczos iteration, as this provides a good approximation for the extreme eigenvalues, for small computational cost. 
If we define these two values as $\alpha$ and $\beta$, and assume $\beta\leq \sigma(L)\leq \alpha$, we may rewrite $L$ as $L=(S\,{\rm Id}-L_sD)$, where $D=(\alpha-\beta)/2$, $S=(\alpha+\beta)/2$, and $-1\leq\sigma(L_s)\leq 1$. We may then expand the exponential of $L_s$ in the Chebyshev polynomials and we
arrive at the following equation for $\varrho$:
\begin{equation}\label{chebexp}
\varrho(t)=e^{-iLt}\varrho_0\approx e^{-itS}\left(\sum_{k=0}^{n_{\rm max}}c_k(t_D) T_k(L_s)\varrho_0\right),
\end{equation}
with $t_D=Dt$.
Both $c_k(t_D)$ and $T_k(L_s)$ can be calculated iteratively:
%\begin{subequations}\label{cktk}
\begin{equation}\label{ck}
 c_k(t)=(2-\delta_{k,0})(-i)^k J_k(t),
 \end{equation}
 \begin{equation}\label{Tk}
T_{k+1}(x)=2T_k(x)x-T_{k-1}(x),
\end{equation}
%\end{subequations}
with initial values $T_0(x)={\rm Id}$, $T_1(x)=x$.
$J_k(t)$ is the $k$-th Bessel function of the first kind.

The Bessel Functions of the First Kind  of integer order may be evaluated directly by using a three-term recurrence relation
\begin{equation}\label{besselj}
J_{n+1}(t)=\frac{2n}{t}J_n-J_{n-1}.
\end{equation}
It is well known that  (\ref{besselj}) becomes numerically unstable for $n>t$, \cite{backward}.
To improve the method, we may exploit the linear nature of the iterative algorithm. It is possible to use  Miller's algorithm, and to solve an inverted form of (\ref{besselj}), i.e. to compute $J_{n-1}$ given $J_n$, $J_{n+1}$ \cite{backward}. When using  Miller's Algorithm it is suggested to expand the number of  terms (providing a sort of buffer), i.e. to start the backward iteration process from $m_{\rm start}=n+r$, where $n$ is the actual order of the function we are interested on and $r$ is some small expansion.  In this case we need to know already from an a priori error analysis how many iterations need to be performed to to get below the threshold $\varepsilon$.

It is possible to prove that for the rescaled Hermitian matrix $L_s$, when applied to a vector of unit Euclidian norm, we have \cite{lubichbok}:
\begin{equation}\label{errlub}
\|P_{m-1}(tL_s)\varrho_0-e^{-it L_s}\varrho_0\|\leq 4(e^{1-(t/2m)^2}\frac{t}{2m})^m, 
\end{equation}
 for $m>t$, where $P_m(t)$ is the order $m$ expansion in Chebyshev polynomials.
This equation indicates that there is a superlinear decay of the error when $m> t$.  The formula may be derived by examining the asymptotic behaviour of the Bessel functions (\ref{asympt}).
We may then use the relation $4(\exp\{1-(\tau/2m)^2\}\frac{\tau}{2m})^m\leq \varepsilon$ to approximate $m$.

From a practical point of view the usual way of applying Chebyshev is to evaluate \cite{spinev,tal}:
\begin{equation}\label{cheb3}
\varrho_{n+1}=e^{-i L dt }\varrho_n\simeq P_{n_{\rm max}}(dt L) \varrho_n,
\end{equation}
where:
\begin{equation}
P_{n_{\rm max}}=e^{-idt S} \sum_{k=0}^{n_{\rm max}}c_k(dt D) T_k(L_s).
\end{equation}
$m$ can be evaluated either from (\ref{errlub}) or directly checking the convergence (\ref{c_max1}). 

To avoid numerical instabilities coming from the iterative formula for the Bessel functions it is also possible to get $P_{m}$ for a given $m$ via the Clenshaw Algorithm \cite{clens,lubichbok}:
\begin{equation}
d_k=c_k \varrho_n+2 L_s d_{k+1}-d_{k+2}, \, k=m-1,m-2,\ldots,0,
\end{equation}
with initial values $d_{m+1}=d_m=0$, and $P_m(dt L)\varrho_0=d_0-d_2$.

\subsection{The Zero--Track--Elimination method}
Nuclear Magnetic Resonance (NMR) is a spectroscopy technique that exploits the interaction between nuclear spins and electromagnetic fields in order to analyse the samples. The temporal evolution of such a system is described via a density matrix that has size $(2I+1/2)^n$ where $I$ is the spin and $n$ the number of nuclei. The exponential growth of the size of $\varrho$ with respect to $n$ impedes the use of simulations when dealing with systems involving more than few ($5$-$6$) spins. Many attempts have been made to solve this (see Refs.~\onlinecite{simpson,kuprov} for recent approaches), even using Chebyshev polynomials \cite{spinev}.
These algorithms have been developed to simulate both liquid systems, where the Hamiltonian is generally time independent, and for Solid--State NMR.
In the last case the Hamiltonian is time dependent due to the non averaging out of  anisotropic interactions during the motion of the sample, and as previously remarked DEC is not applicable in this case.

Recently also a new method for model reduction in the simulation of large spin systems, so called Zero Track Elimination (ZTE),  has been presented, \cite{kuprov}.  This technique is based on the idea of pruning out the elements of $\varrho(t)$ which do not belong to $K(L,\varrho_0)$.
In practice, the initial density matrix $\varrho_0$ is a sparse vector.

In order to reduce the steps needed to evolve the full system, we monitor the elements of  $\varrho(t)$ that stay below a chosen threshold $\xi$ during this first evolution steps and introduce structural zeros based on these observations.
The evolution is then performed in this reduced state space $(\varrho_Z,L_Z)$. 
The idea is extremely appealing, as once the propagator for $L_z$ is evaluated all the subsequent steps have the cost of a reduced matrix--vector, and it is possible to use  (\ref{prop}) for the reduced system.

The initial time length is set as the inverse of the largest Larmor frequency. The Larmor frequency is a given quantity for each element that depends on the physical property of the nucleus and is the frequency of resonance for a non--interacting spin: $\omega_j^0=-\gamma_j B$, where $-\gamma_j$ is the gyromagnetic ration of the nucleus and $B$ the applied magnetic field.

The time--length of the initial check is then:
\begin{equation}\label{larmin}
\delta t=\frac{1}{t_{\rm lar}}=\frac{2\pi}{min_j\{|\omega_j^0|\}} 
\end{equation}
where $\omega^0_j$ is the Larmor frequency of the $j$-th spin.
The following theorem given in \cite{kuprov} assures that it is possible to prune out from the evolution those states that remain 
exactly $0$ during the first few time steps.
It is possible to prove that for a state $|l\rangle$ there holds the equation:
\begin{equation}\label{ztt}
\langle l|e^{-iLt} |\varrho_0\rangle=0,\, t\in [0,\delta t] \Rightarrow  \langle l|e^{-iLt} |\varrho_0\rangle=0,\,  t\in [0,\infty)
\end{equation}

To prove (\ref{ztt}) it is sufficient to expand $e^{-iL t}$ into a Taylor expansion.
In fact:
\begin{equation}
 \langle l |e^{-iL t}|\varrho_0\rangle=\langle l|\sum_{k=0}^{\infty} L^k \frac{(-it)^k}{k!}| \varrho_0\rangle= \sum_{k=0}^{\infty}\frac{(-it)^k}{k!}  \langle l| L^k\varrho_0\rangle, 
 %\,t\in [0,\delta t],
\end{equation}
can be true for any $t \in [0,\delta t]$ only if $\langle l| L^k|\varrho_0\rangle=0$ for all the $k$, but this means that $\langle l |e^{-iL t}|\varrho_0\rangle=0, tt\in [0,\infty)$.

From a practical point of view few states $\langle l|$ obey (\ref{ztt}), but a much higher number of states will stay \textit{close} to $0$ during the first $j$ time steps, where $j=\delta t/dt$.
So the states $\langle l |$ that are pruned out are those for which:
\begin{equation}\label{thin}
 \langle l |e^{-iL t}|\varrho_0\rangle<\epsilon,\quad t \in [0,\delta t]
 \end{equation}

It is claimed \cite{kuprov} that the error of such an approximation is similar to what would be obtained by considering in the Krylov expansion the contributions coming from high values of $n$ in $L^n\varrho_0$.

The main problem of this approach is the choice of $\delta t$, i.e. the duration of the initial propagation.
In fact if we look at a full diagonalization of $L$ we see that:
\begin{equation}\label{diag}
\varrho(t)=X e^{-iD t}X^{-1}\varrho_0,
\end{equation}
so each element of $\varrho(t)$ can be written as a sum of oscillators vibrating at the different $\lambda_j  \in \sigma(L)$.
\begin{equation}\label{osc}
\varrho_k(t)=\sum_{j=1}^n X_{[k,j]}\mu_j e^{-i\lambda_j t }, \qquad \mu_j=\sum_{l=1}^N  X_{[j,l]}\varrho_0^l.
\end{equation}
From (\ref{osc}) it is clear that to ensure that we are not pruning out the low frequencies modes we would need at least  $\delta t \propto  1/\min_j \{|\lambda_j| \} $, and this could be different from the lowest Larmor frequency $\min_j\{|\omega_j^0|\}$ as this last is a quantity related to the non--interacting spins.

 It is possible to have an estimate of the lowest eigenvalue of $L$ using a technique like the one we used for the Chebyshev expansion (\ref{cheby}). The simple choice of $\delta t$ depending on $\lambda_j$ rather then $\omega_j^0$  is not enough to ensure the validity of (\ref{thin}).

Within this framework we can restate (\ref{thin}) in a simpler form for a one dimension function and prove that:
\begin{theorem}\label{zteth}
Given a function $f(t)=\sum_{j=1}^n e^{-\lambda_j t} \mu_j$ s.t. $|f(t)|<\epsilon$ for $t\in[0,2\pi/\min_j\{\lambda_j\}]$, this is not sufficient to ensure that $|f(t)|<\epsilon$ for $t\in[0,\infty)$.
\end{theorem}
\begin{proof} To prove (\ref{zteth})  we can check that if the $\lambda_j$ are not well separated then it is possible to have a combination of similar frequencies that build up on a total frequency that is the lowest common multiplier of the initial frequencies. In the general form if for one or more $\varrho_k$ we have the resonance condition:
\begin{equation}\label{examp}
\varrho_k(t)=\alpha e^{-i2\pi\lambda t}-\sum_j \beta^j e^{-i2\pi(\lambda +\delta^j)t},
\end{equation}
and $\beta_j,\alpha >0$ rationals such that $\alpha=\sum_j  \beta_j$. $\exists$ small enough $\delta_j$ s.t. we have that $|\varrho_k(t)|<\epsilon$,  $t \in[0,\delta t]$ but  at the periodic maximum  $ t >\delta t$ s.t. $|f(t)|=\alpha+\sum_j \beta_j \gg \epsilon$.
\end{proof}

In the appendix we present one example of such a function.

In order to ensure the validity of (\ref{thin}) even with $\delta t$ that depends on $\sigma(L)$ rather than the Larmor frequencies, we need to check the separation of the eigenvalues, and obviously an analysis of that kind would be as expensive as a whole simulation. 

In practice, it might be argued that we are unlikely to have a dynamics like (\ref{examp}), but especially with systems with many spins, a wide variety of behaviors are certainly possible in the general situation.

Focusing on NMR simulations however the reason why ZTE performs so well \cite{kuprov} lies in the fact that the numerical comparison with other methods is not performed on $\varrho(t)$ but on the observables (\ref{obs}).
In particular for experimental reasons the only quantity that can be compared with experiments is the {\em free induction decay} (FID) signal:
\begin{equation}\label{fid}
f(t)={\rm Trace } \left\{\varrho(t) I_p\right\}.
\end{equation}
$\varrho(t)$ can be written as combination of Pauli matrices, and $I_p$ is the shift up operator: $I_p=I_x+iI_y$. 
The Fourier transform of $f(t)$ gives then the spectrum of the sample, where each resonance frequency is revealed by a peak.
Due to relaxation effects in the experiments the shape of the resonance peaks is not a delta function (as would result from (\ref{osc}) but is a smooth function.
To model this smoothness also for the simulated data it is common practice to evaluate the Fourier transform not of (\ref{fid}) but for an exponentially decaying function $\tilde{f}(t)$:
\begin{equation}\label{exp}
\tilde{f}(t)=e^{-\xi t} f(t),
\end{equation}
where the parameter $\xi$ comes from other fitted data \cite{Levitt}.
The main effect of (\ref{exp}) is exactly to smooth out all the low frequency modes that will differentiate $\varrho(t)$ from  $\varrho_Z(t)$.

There are however some drawbacks: 
\begin{itemize}
\item for this method there is no available convergence theory;
\item the performance depends strongly from the initial condition $\varrho_0$, and on $H$. As expected, in our tests the size of the reduced system could change by a factor $2$ depending on the number of interacting spins. The reason for this effect comes from the fact that the less sparse $L$ is, the more non--zero states will appear within the first steps, thus reducing the effectiveness of the reduction $L_z$.
\item Another reason of the strong dependence of ZTE with respect to the initial conditions comes from (\ref{larmin});  depending on the Larmor frequencies, and on the timestep size the number of evolution steps at the beginning can become large, and this being the most expensive part of the simulation the influence of it on the total computation costs can become significant.
\end{itemize}

\section{The Direct Expectation values via Chebyshev}
The common point of all the methods presented in the previous section is that they involve the propagation of the matrix $\varrho$, and for this reason 
they suffer from requiring that matrix operations (or matrix-vector operations) be performed at each step of calculation.
Let us recall that the only quantities that can be compared between quantum simulations and experiments are the observables, for an operator $\hat{Q}$ we have:
\begin{equation}\label{obs1}
 \langle \hat{Q}(t) \rangle={\rm Trace} \{\varrho(t) \hat{Q}\}.
\end{equation}
The first step of DEC is to perform a Chebyshev expansion as seen in section \ref{cheby}, to get:
\begin{equation}\label{chebexp1}
\varrho(t)=e^{-iLt}\varrho_0\approx e^{-itS}\left(\sum_{k=0}^{n_{\rm max}}c_k(t_D) T_k(L_s)\varrho_0\right),
\end{equation}
If we insert (\ref{chebexp1}) into (\ref{obs1}) we find:
\begin{equation}\label{obs_cheb}
\langle \hat{Q}(t)\rangle={\rm Trace}\left\{ \left(e^{-i t S}\sum_{k=0}^{n_{\rm max}}c_k(t_D) T_k(L_s)\varrho_0\right) \hat{Q}\right\}.
\end{equation}
By exploiting the linearity of the trace operation we can pull out of the trace all time dependent parts, and evaluate once 
for all the coefficients $T_k(L_s)$.
In fact we may rewrite  (\ref{obs_cheb}) as:
\begin{equation}\label{obs_cheb2}
 \langle \hat{Q}(t)\rangle=e^{-itS}\sum_{k=0}^{n_{\rm max}}c_k(t_D) {\rm Trace}\left\{ T_k(L_s)\hat{Q}\right\}.
\end{equation}
This is the key equation of the DEC method as it is possible to store an array of scalar values $\tilde{T}_k={\rm Trace}\{(T_k\varrho_0)\hat{Q}\}$. All the time dependent terms are just scalar values that have to be multiplied by $\tilde{T}_k$ to get the evolution of $\hat{Q}$ at any time:
\begin{equation}\label{obs_cheb3}
 \langle \hat{Q}(t)\rangle=e^{-itS}\sum_{k=0}^{n_{\rm max}}c_k(tD) \tilde{T}_k.
\end{equation}

If more than one observable is required it is still possible to use DEC. The only difference with the single expectation case is that we need to store different sets of $T_k$, one for each operator $\hat{Q}$.

\subsection{Stopping Criterion}
The number of terms for the polynomial expansion in (\ref{chebexp}) depends on a prescribed tolerance $\varepsilon$, and on the time $t_D$. In other methods based on Chebyshev approximation \cite{spinev}, the following has been suggested as a stopping criterion:
\begin{equation}\label{c_max}
n_{\rm max}\quad {\rm s.t.}\quad \|c_{n_{\rm max}}(t_D)\|<\varepsilon.
\end{equation}
Due to the zeros of the Bessel function $J$, at fixed time $t_D$, (\ref{c_max}) may hold for some $n$, even though the expansion has not yet reached the convergence regime; it may happen that for $n_1>n$ we have that $c_{n_1}(t_D)>c_n(t_D)$.
To avoid this effect it is enough to use as a stopping criterion a combination of two Bessel functions; the cost of such a stopping criterion is that at most we need to perform an extra iteration step (\ref{ck}).
In our numerical tests we have used the following:
\begin{equation}\label{c_max1}
 n_{\rm max}\, {\rm s.t.}\, \sqrt{\|c_{n_{\rm max}-1}(t_D)\|^2+\|c_{n_{\rm max}}(t_D)^2\|}<\varepsilon.
\end{equation}
The total time $\tau$ plays a role here, since the larger $\tau$ the more terms $(T_k,c_k)$ will be needed to get $| c_k|$ below
 the threshold $\varepsilon$.

\subsection{Computation of the Expansion}
In order to optimise the number of terms we evaluate, but without having to check at each step whether we have already evaluated enough terms $T_k$,  we propose to evaluate first $\langle \hat{Q}(t)\rangle$, at the final time $\tau$, and to store the $N_{\rm max}$ values of $\tilde{T}_k$.
We can prove that:
\begin{theorem}
Given a Chebyshev expansion for an exponential $e^{-iL t}\varrho_0$ if (\ref{c_max1}) holds for a given time $\tau$ and small enough $\varepsilon$ then (\ref{c_max1}) holds for any time $t\leq \tau$. 
\end{theorem}
\begin{proof} From equation (\ref{ck}) it is clear that $c_k$ depends on the Bessel functions.  If we look at the asymptotic behaviour of the Bessel function of fist kind, for any $k\in N$, we have that, for $k$ fixed \cite{handbok}:
\begin{equation}\label{asympt}
J_k(t)\sim \frac{1}{\Gamma(k+1)}\left(\frac{t}{2}\right)^k, \qquad  \lim t\rightarrow 0,
\end{equation}
where $\Gamma(t)$ is the Euler--$\Gamma$ and for $n\in Z$ we have that $\Gamma(n)=(n-1)!$.
Equation (\ref{asympt}) shows that for any $k\neq 0$, in a neighbourhood of $t=0$, $J_k(t)$ is increasing monotonically with respect to $t$.
This behaviour is maintained for the whole interval $[0,j_k']$ where $j_k'$ is the first zero of the derivative of $J_k$.
It is possible to show (see Ref.~\onlinecite{handbok},  Eq.$9.5.2$), that $k\leq j_k'$;  consequently we can say that if (\ref{c_max1}) holds for a given $n_{\rm max}$ at $\tau$ and $\tau\leq n_{\rm max}$, then we are in the monotonically increasing region for $J_{n_{\rm max}}$ and $J_{n_{\rm max}+1}$.  In this case, equation (\ref{c_max1}) holds also for any $t\leq \tau$.
 \end{proof}

%\begin{wrapfigure}{c}{0.6\textwidth}
%\centering
%\includegraphics[width=0.6\textwidth]{figure1.eps}
%\caption{Example of few integer order Bessel Functions of the first kind.}
%\noindent
%%\hrulefill
%\label{fig:bessel}
%\end{wrapfigure} 

 \begin{figure}
 \begin{centering}
 \includegraphics[width=0.45\textwidth]{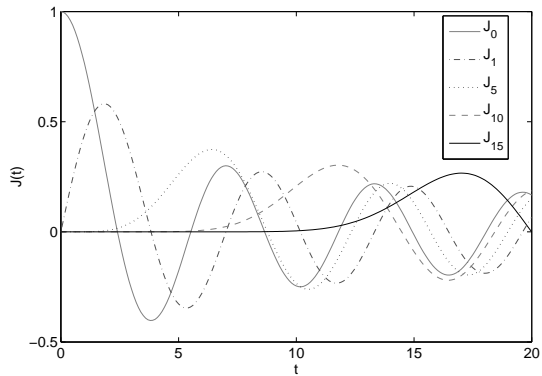}
 \caption{Example of few integer order Bessel Functions of the first kind. } 
 \label{fig:bessel}
 \end{centering}
 \end{figure}

\subsection{Efficient Implementation}
The cost of DEC is all in the first step.   Note that the cost of the evaluation of any $\tilde{T}_k$ itself is roughly equivalent to that of 
a matrix--matrix multiplication, as per the iteration $T_{k+1}(L_s)=2T_k(L_s)-T_{k-1}$ (\ref{Tk}). But what is actually needed in all our calculations is $T_k\varrho_0$. Because of the linearity of the iterative expression, we may multiply $T_0$ and $T_1$ by $\varrho_0$ and then use (\ref{Tk}) directly on $T_k\varrho_0$. The iterated operation in then just a matrix--vector multiplication.

After all the $\tilde{T}_k$ needed have been stored, it is possible to get $\langle Q(t)\rangle$ at any time $t\in[0,\tau]$ by evaluating:
\begin{equation}\label{Qt}
\langle Q(t)\rangle=e^{-i Dt_s} \sum_{k=1}^{m_{max}} c_k(t_s) \tilde{T}_k
\end{equation} 
where the $c_k$ are evaluated iteratively via (\ref{ck}), and $m_{max}$ satisfies (\ref{c_max1}) for $tD$.\\
For the details of the algorithm we refer to the Appendix.
As a final remark we note that if the Hamiltonian is time dependent it is not possible to apply DEC, as it is not possible anymore to isolate the time dependent part out of the trace.

\section{Numerical Experiments}

As in many other physical systems, nuclear spin dynamics provides a perfect example to test DEC, because the final outcome of the simulations is an observable, the free induction decay (FID) signal, and this result is the sole important quantity, as it is the only data available from experiment.

As Hamiltonian we assumed a sum of isotropic chemical shift and the isotropic term of a pair interaction called Homonuclear  J--couplings, that depends on the inner product $I_j\cdot I_k$ \cite{Levitt}:
\begin{equation}
  H=-\sum_{j=1}^n \omega_j^0 I_j^z+\sum_{j,l=1}^n J_{jl} I_j\cdot I_l.
\end{equation}
For the initial density matrix we set $\varrho_0= -I_y$, that is the result of the application of a so called  $x$-pulse to a sample already under the effect of a strong constant magnetic field along the $z$ direction \cite{Levitt}. This is the usual initial condition when the acquisition of  the signal starts.

An illustration of the structure of the Liouvillian matrix is presented in Fig.\ref{fig:spy}. The sparsity depends on the number of interactions among the spins. In most cases the J--coupling interaction matrix $J$ is relatively sparse. In our numerical test  a strong coupling system, where $J$ is dense and each spin interacts with every other spins, was simulated.

%\begin{wrapfigure}{c}{0.8\textwidth}
\begin{figure}
\centering
\includegraphics[width=0.45\textwidth]{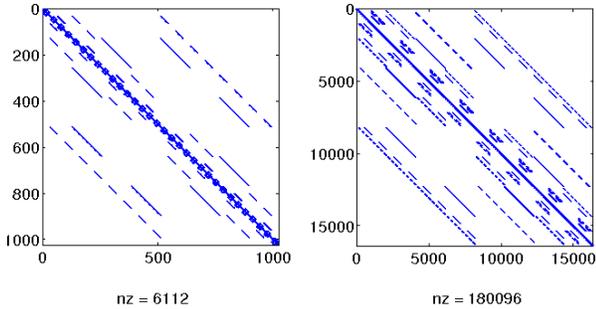}
\caption{Left: Sparse structure of the Liouvillian ($n=1024$, $nz=6112$) for a system of $5$ spins. Right: Structure of the Liouvillian ($n=16384$, $nz=180096$) for a system of $7$ spins. Both the systems are in a weak coupling condition, each spin interacts with approximately half the other spins.}
\noindent
\label{fig:spy}
%\end{centering}
%\end{wrapfigure} 
\end{figure}
%
%\begin{figure}
%\begin{centering}
%\includegraphics[width=0.8\textwidth]{figure2.eps}
%\caption{Left: Sparse structure of the Liouvillian ($n=1024$, $nz=6112$) for a system of $5$ spins. Right: Structure of the Liouvillian ($n=16384$, $nz=180096$) for a system of $7$ spins. Approximately each spin is interacting with half the other spins in both cases.}
%\noindent
%\label{fig:spy}
% \end{centering}
%\end{figure}

Due to the fact that our implementation involves only matrix--vector multiplication, techniques developed both for structured and unstructured sparse matrices may be exploited.

For comparison of computational costs we tested this method with an increasing number of spin particles using different methods to evaluate the exponential.
In particular to examine the error we compared DEC with the {\sl expm} function of \textsl{MATLAB}, that uses a scaling and squaring algorithm with  Pade' approximation.  In this way we evaluate once for all $U=e^{-iL dt}$ where $dt$ is the step size of the simulation, and then at each time--step we propagate $\varrho$:
\begin{equation}\label{prop}
 \varrho_{n+1}=U\varrho_n.
\end{equation}
It is well known that in terms of computational costs this simplistic approach performs poorly, so we  compared DEC also with the other methods presented in this paper:  Krylov expansion via Lanczos \cite{sidje},  the Chebyshev expansion, and ZTE \cite{kuprov}.

For the Lanczos method we have used the function \textit{expv} of the package EXPOKIT \cite{sidje} written in \ textsl{MATLAB}.

\subsection{Summary of Results}

%\begin{wrapfigure}{r}{0.6\textwidth}
\begin{figure}
\begin{centering}
\includegraphics[width=0.45\textwidth]{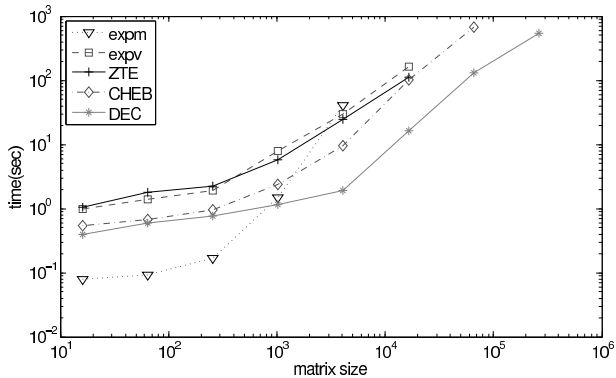}
\caption{Logarithmic comparison of computational costs for Lanczos, Zero Track Pruning (ZTE) and Direct Expectations via Chebyshev (DEC), with $dt=0.1$ $N=1000$.}
\noindent
 \label{fig:loglog}
 \end{centering}
%\end{wrapfigure}
\end{figure}
The error-to-cost (measured in CPU time) diagrams are shown in Figure \ref{fig:loglog} for all the methods described. 
All the numerical tests have been performed on a Dell PowerEdge 1950 wiht 4GB RAM and a DualCore Intel processors running in 32bit mode. The language used is \textsl{MATLAB}.

It is clear that DEC is almost an order of magnitude more efficient than the alternatives. To avoid instabilities coming form the evaluation of the Bessel functions in these numerical tests we set the tolerance to be $\varepsilon = 1e-7$.

DEC performs at its best for short time simulations (i.e. when total time $\tau$ is small), so that we do not need to evaluate a large number of $T_k$, and when at the same time the use of small time step $dt$ is required, as the cost for any step after the first is negligible.
For instance, while for all the other methods the cost of a $1000$ step simulation with $dt=0.1$, is approximately half the cost of a simulation of $1000$ steps with $dt=0.1$, for DEC there is a gain of almost an order of magnitude, see Fig.\ref{fig:loglog1}.

%\begin{wrapfigure}{r}{0.6\textwidth}
\begin{figure}
\begin{centering}
\includegraphics[width=0.45\textwidth]{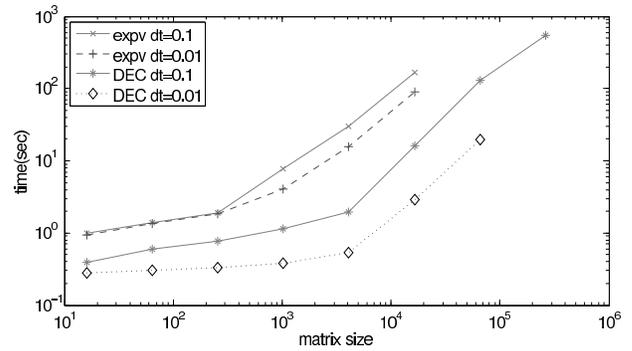}
\caption{Logarithmic comparison of computational costs for Lanczos  and DEC when simulating for the same number of total steps $N$ but with different stepzise $dt$. $N=1000$ in both the cases.}
\label{fig:loglog1}
\end{centering}
%\end{wrapfigure}
\end{figure}

\begin{table}
%\begin{tabular}

\vspace*{0.3cm}
%\begin{tabular}{|c|c|c|c|c|cc|c|}\hline
%$n$ spin &Matrix size &expm & expv & Chebyshev & ZTE & Reduced\footnotemark[1] & DEC \\
%\hline

%2&16     & 0.08 & 0.99  &  0.54    & 1.06 &  8      &  0.39    \\
%3&64     & 0.1 & 1.41   &  0.68   & 1.79   &  30     &  0.60     \\
%4&256    & 0.17  & 1.88   & 0.95   & 2.26  &  112    &  0.76     \\
%5&1024   & 1.5  & 7.87   &  2.40  & 5.80  &  420    &  1.14    \\
%6&4096   & 40.84 & 29.78   &  9.60  & 24.51 &  1584    &  1.93   \\
%7&16384  &      & 165.06 &  102.92 & 112.85  &  6006   &  16.11    \\ 
%8&65536  &      &        &   677.22    &      &         &  129.93    \\ 
%9&262144  &      &        &       &       &         &  542.98    \\ 
%\hline
\begin{tabular}{|c|c|c|c|cc|c|}\hline
Matrix size &expm & expv & Chebyshev & ZTE & Reduced\footnotemark[1] & DEC \\
\hline

16     & 0.08 & 0.99  &  0.54    & 1.06 &  8      &  0.39    \\
64     & 0.1 & 1.41   &  0.68   & 1.79   &  30     &  0.60     \\
256    & 0.17  & 1.88   & 0.95   & 2.26  &  112    &  0.76     \\
1024   & 1.5  & 7.87   &  2.40  & 5.80  &  420    &  1.14    \\
4096   & 40.84 & 29.78   &  9.60  & 24.51 &  1584    &  1.93   \\
16384  &      & 165.06 &  102.92 & 112.85  &  6006   &  16.11    \\ 
65536  &      &        &   677.22    &      &         &  129.93    \\ 
262144  &      &        &       &       &         &  542.98    \\ 
\hline

\end{tabular} 
\footnotetext[1]{Size of $\varrho$ for the reduced system}
\label{table:comp1}
\caption{Comparison of computational costs, CPU time in seconds, for $dt=0.1$, $N=1000$.}

\end{table}

\section{Conclusion}
In this article we have presented a new method for simulation of an observable in a mixed quantum system. 
By expanding the exponential of the Hamiltonian in Chebyshev polynomials, and exploiting the trace operation performed 
when evaluating the expectation value of an observable, it is possible to reduce the evolution of any observable to a one--dimension function that can be evaluated directly.

We have also presented an optimal algorithm to perform such a calculation, and we have shown how this new method can easily compete in term of computational costs with a variety of model reduction approaches, whilst maintaining the approximation errors below a chosen threshold.

Moreover we have also studied the Zero Track Elimination method  that has appeared in the literature.  We have demonstrated that the ZTE method cannot be a reliable method for every system, but can be a powerful alternative in the restricted setting of NMR simulations.

G.M. is very grateful to Arieh Iserles for useful suggestions at the start of this work.

\appendix
\section{The DEC algorithm} 
We provide here a detailed description of the algorithm.\\
\textsl{
{\it inputs}: $L$ hermitian matrix $n \times n$, $\varrho_0$ vector \\
{\it outputs}: expectation value $f(t)$ evaluated at $f(j \Delta t)$, $j=1,\ldots,N$.
\begin{enumerate}
 \item evaluate $\alpha$,$\beta$ via Lanczos s.t. $\alpha\leq\sigma(L)\leq\beta$
\item scale $L$ and get $L_s$
\item evaluate $T_0= {\rm Id}\varrho_0$, $T_1=L_s\varrho_0$
\item {\bf while} $\|c_k\|< \varepsilon$
\item \hspace{1cm} $c_k=(2-\delta_{k,0})(-i)^k J_k(\tau S)$, $\tau=$ total time 
\item \hspace{1cm} $T_{k+1}=2T_kL_s-T_{k-1}$
\item \hspace{1cm} store $\tilde{T}_k={\rm Trace}\{(T_k\varrho_0)\hat{Q}\}$
\item {\bf end}
\item {\bf for} $j=1:{\rm N}$
\item \hspace{1cm} re-evaluate the $c_k$ at different time $t=j d t$ 
\item \hspace{1cm} $f(j)=\sum c_k \tilde{T}_k$
\item {\bf end}
\end{enumerate}
}
\section{A counter--proof of the ZTE elimination}
%\begin{wrapfigure}{r}{0.4\textwidth}
\begin{figure}
\begin{centering}
\includegraphics[width=0.4\textwidth]{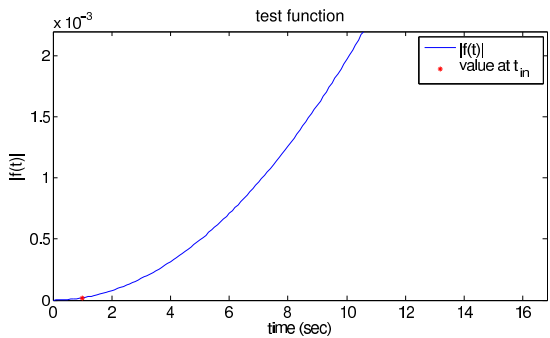}
\caption{$t$ against $|f(t)|$ plot, the dot is the value at time $\delta t$. }
\label{fig:count}
\end{centering}
%\end{wrapfigure}
\end{figure}
We provide a specific counter--example for the Zero Track Elimination.
In Fig.\ref{fig:count} we have plot the absolute value of the function:
\begin{equation}
f(t)=e^{-i2\pi t}-\ha e^{-i 2\pi(1.001)t}-\ha e^{-i 2\pi(0.999)t}.
\end{equation}
If we look at the $\delta t$ evaluated looking at the lowest isolated frequency we have that $\delta t\simeq 1$, within this time $max|f(t)|\simeq 2 \times 10^{-5}$ but clearly the maximum amplitude of this periodic function will be $2$.

\bibliography{plain}

\end{document}